\documentclass[12pt,reqno]{amsart}
\usepackage{amsmath}
\usepackage{etoolbox}
\makeatletter
\patchcmd\maketitle
{\uppercasenonmath\shorttitle}
{}

{}{}
\patchcmd\maketitle
{\@nx\MakeUppercase{\the\toks@}}
{\the\toks@}
{}
{}{}
\patchcmd\@settitle{\uppercasenonmath\@title}{\Large}{}{}
\patchcmd\@setauthors
{\MakeUppercase{\authors}}
{\authors}
{}{}
\makeatother
\usepackage{amsmath, amsthm, amscd, amsfonts, amssymb, graphicx, color}
\usepackage{url}
\usepackage{tikz-cd}
\usepackage{hyphenat}
\everymath{\displaystyle}
\newtheorem{theorem}{Theorem}

\newtheorem{corollary}[theorem]{Corollary}

\newtheorem{definition}[theorem]{Definition}
\newtheorem{example}[theorem]{Example}

\newtheorem{lemma}[theorem]{Lemma}

\newtheorem{proposition}[theorem]{Proposition}
\newtheorem{remark}[theorem]{Remark}

\usepackage[colorlinks=true]{hyperref}  
\hypersetup{urlcolor=blue, citecolor=red, linkcolor= blue}

\begin{document}
\keywords{$q$-joint numerical range, $q$-joint numerical radius, operator norm, positive operator, $A,q$-numerical range}
\subjclass[2020]{47A63, 47A12, 47A05, 47A30.}
\title[Joint $q$-Numerical Ranges of Operators in Hilbert and Semi-Hilbert Spaces]
{Joint $q$-Numerical Ranges of Operators in Hilbert and Semi-Hilbert Spaces}

\author[K. Feki, A. Patra, J. Rani and Z. Taki] {{ \large Kais Feki}$^{1}$, {\large Arnab Patra}$^{2}$, {\large Jyoti Rani}$^{3}$ and {\large Zakaria Taki}$^{4}$}

\address{$^{[1]}$ Department of Mathematics, Faculty of Science and Arts, Najran University, Najran $66462$, Kingdom of Saudi Arabia.}
\email{\url{kfeki@nu.edu.sa}}
\address{$^{[2]}$ Department of Mathematics, Indian Institute of Technology Bhilai, Bhilai 491002, India}
\email{\url{arnabp@iitbhilai.ac.in}}

\address{$^{[3]}$ Department of mathematics, Indian Institute of Technology Bhilai , Bhilai 491002, India.
}
\email{\url{jyotir@iitbhilai.ac.in}}
\address{$^{[4]}$ Department of Mathematics,
		Faculy of Sciences-Semlalia,
		University Cadi Ayyad,
		Av. Prince My. Abdellah, BP: 2390,
		Marrakesh (40.000-Marrakech), Morocco (Maroc)}
\email{\url{zakaria.taki2016@gmail.com}}


\date{\today}

\maketitle

\begin{abstract}
This paper introduces and investigates the concept of the $q$-numerical range for tuples of bounded linear operators in Hilbert spaces. We establish various inequalities concerning the $q$-numerical radius associated with these operator tuples. Furthermore, we extend our study to define the $q$-numerical range in semi-Hilbert spaces and provide a proof of its convexity. Additionally, we explore several related results in this context.
\end{abstract}

\section{Introduction and Preliminaries}
The numerical range plays a significant role in understanding the behavior and characteristics of linear operators and matrices. It finds applications in various fields such as stability analysis of dynamical systems, optimization problems, control theory, and quantum mechanics. Apart from the classical numerical range, there are several extensions, including the algebraic numerical range, $q$-numerical range, and joint numerical range. To gain a more comprehensive understanding of these concepts, we recommend referring to the following sources: \cite{baklouti2018joint,buoni1978joint,gustafson1997numerical,gau2021numerical,moore1969adjoints,chan2018note,cho1981boundary}.

Let $\mathbb{N}^*$ denote the set of all positive integers. For $d \in \mathbb{N}^*$, we denote by $\mathbb{C}^d$ the set of $d$-tuples of complex numbers. For a vector $\xi = (\xi_1, \dots, \xi_d) \in \mathbb{C}^d$, we define $\overline{\xi} = \left(\overline{\xi_1}, \dots, \overline{\xi_d}\right)$ as the vector formed by taking the conjugate of each component of $\xi$. The infinity norm and the Euclidean norm of a vector $\xi = (\xi_1, \dots, \xi_d) \in \mathbb{C}^d$ are defined respectively as:
\begin{align*}
\Vert \mathbf{\xi} \Vert_\infty:=\max\limits_{1\leq i \leq d }\vert{\xi_i}\vert \  \textup{ and } \  \| \mathbf{\xi} \|_2:=\left(\displaystyle \sum_{i=1}^{d}|\xi_i|^{2}\right)^{\frac{1}{2}}.
\end{align*}
It is widely known that for any $\xi \in \mathbb{C}^{d},$ the following inequalities hold
	\begin{equation}\label{eq1}
		\Vert \mathbf{\xi} \Vert_\infty \leq \Vert \mathbf{\xi} \Vert_2 \leq \sqrt{d}\Vert \mathbf{\xi} \Vert_\infty.
	\end{equation} 

Let $\mathcal{B(H)}$ be the $C^*$-algebra of all bounded linear operators acting on a complex Hilbert space $(\mathcal{H},$ $\langle \cdot, \cdot \rangle)$. The identity operator on $\mathcal{H}$ is denoted by $I$, and the norm resulting from the inner product $\langle \cdot \rangle$, is denoted by $\Vert \cdot \Vert$.  The adjoint, real part, imaginary part, range, and null space of an operator $T\in \mathcal{B(H)}$ are denoted, respectively, by $T^*,$ $\Re(T),$ $\Im(T),$ $\mathcal{R}(T)$, and $\mathcal{N}(T)$. Note that $\Im(T)=\frac{T-T^*}{2i}$ and $\Re(T)=\frac{T+T^*}{2}$.

 For a $d$-tuple $\mathbf{T}=(T_1, \dots, T_d)$ of elements of $\mathcal{B(H)},$ we use the following   notations  $\mathbf{T}^{*},$  $ \Re(\mathbf{T}),$  $\Im(\mathbf{T}),$  $\mathbf{T}x$ and $\langle \mathbf{T} x,y \rangle$  to denote respectively  the vectors  $(T_{1}^{*}, \dots, T_{d}^{*}),$ $ \big(\Re(T_1),\dots, \Re(T_d) \big),$ $ \big(\Im(T_1),\dots, \Im(T_d)\big),$ $\big(T_1x, \dots,T_dx\big)$ and  \\ $\big(\langle T_1 x,y \rangle, \dots,\langle T_d x,y \rangle\big),$ where $x,y \in \mathcal{H}$.  We also use the symbols  $\| \mathbf{T}x\|$  and  $\| \langle \mathbf{T} x,y \rangle\| $  to designate respectively the quantities  $\big \Vert  \big(\|T_1x\|,\dots,\|T_d x\| \big) \big \Vert_{2} $ and $
\big\| \big(\langle T_1 x,y \rangle, \dots ,\langle T_d x,y \rangle \big)\big\|_2$ for every $x,y \in \mathcal{H}$.

	The set of all positive operators acting on $\mathcal{H},$ denoted by $\mathcal{B(H)}^{+},$ is given by
\begin{align*}
\mathcal{B(H)}^{+}:=\{ T\in \mathcal{B(H)} : \langle Tx, x \rangle \geq 0 \ \text{ for all } \ x\in \mathcal{H} \}.
\end{align*}	
For a positive operator $T \in \mathcal{B(H)}^{+},$ we designate its square root by $T^{\frac{1}{2}}$ its square root.
	
Recall that  an operator $T \in \mathcal{B(H)}$ is said to be unitary(respectively self-adjoint) if $T^{*}T=TT^{*}=I$(respectively $T^*=T$).
	
	 For an operator $T \in \mathcal{B(H)}$, the numerical range $W(T)$, the  numerical radius $w(T)$ and the  operator norm $\|T\|$ of $T$, are defined, respectively, by:
\begin{align*}
W(T)&:= \{\langle Tx,x \rangle  : x \in \mathbb{S}^1\}, \\
\omega(T)& := \sup\{| \langle Tx,x \rangle | : x \in \mathbb{S}^1\}
\end{align*} 
and
\begin{align*}
\|T\|:= \sup\{| \langle Tx,y \rangle | : x,y \in \mathbb{S}^1\},
\end{align*}
where $\mathbb{S}^1$ denotes the unit circle in $\mathcal{H}$, consisting of all vectors $x \in \mathcal{H}$ with $\|x\|=1$.

	It is well known that $w(\cdot)$ defines a norm on $\mathcal{B(H)}$, which is equivalent to the usual operator norm $\| \cdot \|$. In fact, for every $T \in \mathcal{B(H)}$, we have the following inequalities
	\begin{equation*}
		\frac{\|T\|}{2} \le w(T) \le \|T\|.
	\end{equation*}

The classical numerical range has been extended in several directions. Among these extensions we cite the joint numerical range, the $q$-numerical range and the $A$-numerical range.   Over the years, these   sets have been studied extensively by many researchers, see for instance \cite{li1998some, chien2007q,zamani2019numerical,feki2020tuples}.

	Recall that   the joint numerical range 	$JtW(\mathbf{T})$, joint numerical radius $Jt\omega(\mathbf{T})$ and norm $\Vert{\mathbf{T}}\Vert$ of a $d$-tuple $\mathbf{T}=(T_1, \dots, T_d)$ of elements of $\mathcal{B(H)}$   are  respectively defined as follows:
\begin{align*}
	JtW(\mathbf{T})&:=\{ \langle \mathbf{T}x, x\rangle :  x \in \mathbb{S}^1\}, \\
	Jt\omega(\mathbf{T})& :=\sup\{\Vert \langle \mathbf{T}x, x\rangle  \Vert_2 : x \in \mathbb{S}^1\}=\sup_{x \in \mathbb{S}^1}\left\lbrace \left(\sum_{i=1}^d| \langle T_ix,x \rangle|^2\right)^\frac{1}{2} \right\rbrace,
\end{align*}
and
\begin{equation*}
	\Vert{\mathbf{T}}\Vert:=\sup\{\Vert  \mathbf{T}x \Vert_2 :  x \in \mathbb{S}^1\}=\sup_{x \in \mathbb{S}^1}\left\lbrace \left(\sum_{i=1}^d\|T_ix \|^2\right)^\frac{1}{2} \right\rbrace.	
\end{equation*}

The convexity of $JtW(\mathbf{T})$ is not assured in general; see \cite[p.359]{gau2021numerical}. Later studies have investigated the specific conditions under which $JtW(\mathbf{T})$ can be convex. A sufficient condition was given in \cite[p.360]{gau2021numerical}: $JtW(\mathbf{T})$ is convex if either $\dim \mathbf{Sp}\{T_1, T_2, \ldots, T_d, I\} \le 3$ or $\dim(\mathcal{H}) > 3$ and $\dim \mathbf{Sp}\{T_1, T_2, \ldots, T_d, I\} \le 4$. Here, $\mathbf{Sp}M$ refers to the space spanned by $M$.

It has also been shown that for a $d$-tuple of self-adjoint operators $\mathbf{T} = (T_1, \ldots, T_d)$ on an infinite-dimensional space, the joint numerical range $JtW(\mathbf{T})$ forms a star-shaped set \cite{li2009joint}. Further studies can be found in \cite{li2000convexity, gutkin2004convexity, li2009joint, baklouti2018joint}.

In everything that follows, let $q \in \mathbb{C}$ with $|q| \leq 1$ and $\dim \mathcal{H} \geq 2$. Recall that the $q$-numerical range $W_q(T)$ and $q$-numerical radius $\omega_q(T)$ of an operator $T \in \mathcal{B(H)}$ are defined as follows:
\begin{align*}
W_q(T) & := \{\langle Tx, y \rangle : (x, y) \in \mathbb{S}_q(\mathcal{H})\},
\end{align*}
and
\begin{align*}
\omega_q(T) := \sup\{|\langle Tx, y \rangle| : (x, y) \in \mathbb{S}_q(\mathcal{H})\},
\end{align*}
where $\mathbb{S}_q(\mathcal{H})$ is a subset of $\mathcal{H} \times \mathcal{H}$ defined by:
\begin{equation*}
\mathbb{S}_q(\mathcal{H}) := \{(x, y) \in \mathcal{H} \times \mathcal{H} : \|x\| = \|y\| = 1 \textup{ and } \langle x, y \rangle = q\}.
\end{equation*}
It is well-known that $W_q(T)$ is a non-empty subset of $\mathbb{C}$; see \cite{li1994generalized}. Additionally, Tsing proved in \cite{r.tsing} that $W_q(T)$ is always convex. He also established that $W_q(T)$ satisfies the following equality:
\begin{equation*}
W_q(T) = \{z \in \mathbb{C} : |z - \langle Tx, x \rangle| \leq \sqrt{1 - |q|^2} \left(\|Tx\|^2 - |\langle Tx, x \rangle|^2\right)^\frac{1}{2} \textup{ for some } x \in \mathcal{H}\}.
\end{equation*}
It is easy to verify that $W_q(U^*TU) = W_q(T)$ for any unitary operator $U$ on $\mathcal{H}$, and $W_{\mu q}(T) = \mu W_q(T)$ for any $\mu \in \mathbb{C}$ with $|\mu| = 1$.

One of the main goals of this work is to introduce an extension of the $q$-numerical range to a multivariable setting. Specifically, we will define the joint $q$-numerical range for a $d$-tuple of operators $\mathbf{T}=(T_1, T_2, \dots, T_d)$ acting on $\mathcal{B(H)}^d$, denoted as $JtW_q(\mathbf{T})$. In this paper, our objective is to investigate the convexity of $JtW_q(\mathbf{T})$. Additionally, we establish connections between the joint eigen-spectrum, joint approximate spectrum, and the joint $q$-numerical range. Moreover, we examine the relationship between the joint $C$-numerical range and the joint $q$-numerical range. We also establish various inequalities concerning the $q$-numerical radius associated with these operator tuples. Moreover, our research delves into defining the $q$-numerical range in semi-Hilbert spaces, which are Hilbert spaces augmented with a semi-inner product structure induced by a positive operator $A$ on $\mathcal{B(H)}$. We also prove its convexity and explore related results in this setting.

\section{Joint $q$-numerical range and radius of operators}
In this section, we aim introduce and  investihate the concept of the $q$-numerical range for tuples of bounded linear operators in Hilbert spaces. 
We then also introduce the so-called joint  $q$-numerical radius  associated with these operator tuples and 
We establish various inequalities concerning it. Let start with the following definition which is a extension of the $q$-numerical range in single variable operator theorry.
\begin{definition}
Let $\mathbf{T} = (T_1, T_2, \dots, T_d)\in \mathcal{B}(\mathcal{H})^d= \underbrace{\mathcal{B}(\mathcal{H}) \times \mathcal{B}(\mathcal{H}) \times \cdots \times \mathcal{B}(\mathcal{H})}_{d\text{-times}}
	$. The joint $q$-numerical range of $\mathbf{T}$ is defined as follows:
\begin{equation*}
JtW_q(\mathbf{T}) := \left\{ \left(\langle T_1 x, y \rangle, \dots, \langle T_d x, y \rangle \right) \in \mathbb{C}^d : (x, y) \in \mathbb{S}_q(\mathcal{H}) \right\}.
\end{equation*}
\end{definition}
It should be noted here that When $q = 1$, the joint $q$-numerical range simplifies to the classical joint numerical range. Hence, the concept of joint $q$-numerical range provides a valuable generalization of the classical joint numerical range.

It is clear that $JtW_q(\mathbf{T})$ is a non-empty, bounded subset of $\mathbb{C}^d$. However, it is generally not closed unless $\mathcal{H}$ is a finite-dimensional space. Additionally, a straightforward calculation shows that the joint $q$-numerical range satisfies the following properties:
\begin{align*}
JtW_{e^{i\theta} q}(\mathbf{T})&=e^{i\theta} JtW_q(\mathbf{T}),\\ 
 JtW_{\overline{q}}(\mathbf{T}^*)&=\Big\{ \overline{\xi} \in \mathbb{C}^{d} : \xi \in  JtW_q(\mathbf{T})  \Big\}
\end{align*}
  and 
\begin{align*}
JtW_q(\alpha \mathbf{T}+ \beta \mathbf{I})=\alpha JtW_q(\mathbf{T})+(\beta q,\dots,\beta q)
\end{align*}  
 for every $\alpha,\beta\in\mathbb{C}$ and every $\theta \in \mathbb{R}$, where $\mathbf{I}=(I,\dots, I).$ 

Unlike the case $d=1$, the joint $q$-numerical range $JtW_q(\mathbf{T})$ is generally not convex, as demonstrated by the following example.
\begin{example}\label{Example Intro}
  Let $z_{1}, z_{2}$ be two fixed non-zero complex numbers  and consider the following $2\times 2$ matrices  $T_1=\begin{bmatrix}
		z_1 & 0\\
		0 &0
	\end{bmatrix}$ and $T_2=\begin{bmatrix}
		0 & 0\\
		z_2 &0
	\end{bmatrix}.$  A simple calculation proves that 
		\begin{align*}
		JtW_q(T_1,T_2)=\left \lbrace \big(z_1x_1\overline{y_1},z_2x_1\overline{y_2}\big):  \sum_{i=1}^{2} \vert x_i \vert ^{2}= \sum_{i=1}^{2}  \vert y_i \vert ^{2}=1   \text{ and } x_1\overline{y_1}+ x_{2} \overline{y_2}=q \right \rbrace.	
		\end{align*}
In the case when $q \ne 0$, we have  $(0,0) \in JtW_q(T_1,T_2) $ and $(z_1q,0)\in JtW_q(T_1,T_2),$ but the convex combination $t (z_1q,0)+(1-t)(0,0)= (z_1 t q,0) \notin JtW_q(T_1,T_2),$ where $ t \in (0,1)$.  For the other case, when $q=0,$ we have    
$$(\frac{\sqrt{3}}{4}z_1,\frac{-1}{4}z_2) ,(\frac{\sqrt{3}}{4}z_1,\frac{3}{4}z_2)\in JtW_0(T_1,T_2),$$ but  
$$\frac{3}{4}(\frac{\sqrt{3}}{4}z_1,\frac{-1}{4}z_2)+ \frac{1}{4}(\frac{\sqrt{3}}{4}z_1,\frac{3}{4}z_2)=(\frac{\sqrt{3}}{4}z_1,0) \notin JtW_0(T_1,T_2).$$ 
Hence, from the previous cases, we deduce that  $JtW_q(T_1,T_2)$ is not convex.
\end{example}

Some properties of the $q$-joint numerical range are summurized in the following proposition.

\begin{proposition}\label{prop1}
	Let $\mathbf{T}=(T_1,T_2,\dots,T_d)\in \mathcal{B(H)}^d$. Then
	\begin{enumerate}
	
\item[(i)] For any unitary operator $U \in \mathcal{B(H)},$ we have    $JtW_q(\mathbf{U}^*\mathbf{T}\mathbf{U})=JtW_q(\mathbf{T}),$ where $\mathbf{U}^*\mathbf{T}\mathbf{U}$ is the $d$-tuple of operators given by: $$\mathbf{U}^*\mathbf{T}\mathbf{U}=(U^*T_1U, \dots, U^*T_dU).$$	
	\item [(ii)] $JtW_{ q}(\mathbf{T})$ is convex if and only if $JtW_{ q}(\mathbf{T},I)$ is convex.	Here $(\mathbf{T},I)$ is a $(d+1)$-tuple of operators in $\mathcal{B}(\mathcal{H})^{d+1}$.			
	\end{enumerate}
	
\end{proposition}
\begin{proof}
The assertion (i) follows directly from the definition of the  joint $q$-numerical range.  For assertion (ii), it suffices to use the fact that  $$JtW_q({\mathbf{T},I})=JtW_q(\mathbf{T}) \times \{q\}.$$
\end{proof}

In the following theorem, we establish a convexity criteria for $JtW_q(\mathbf{T})$ utilizing the spanning set $\mathbf{Sp}\{T_1,  \dots, T_d\}$.

\begin{theorem}\label{THMPRIN}
Let $\mathbf{T}=(T_1,\dots,T_d)\in  \mathcal{B(H)}^d.$ Then, the following assertions are equivalent:
\begin{enumerate}
\item[(i)] $JtW_q(\mathbf{T})$ is convex,
\item[(ii)] $\mathbf{Sp}\{T_1,\dots,T_d\}$ is  generated by a family $\{ R_1,\dots, R_n\}$   $(n\in \mathbb{N}^{*})$ such that  $JtW_q(R_1,\dots,R_n)$ is convex,
\item[(iii)] for each  $S_1,\dots,S_m \in \mathbf{Sp}\{T_1,T_2,\dots,T_d\}$ $(m\in \mathbb{N}^{*}),$ the set  $JtW_q(S_1,\dots,S_m)$ is convex.
\end{enumerate}
\end{theorem}
 
\begin{proof}
Note that the implications $(iii)\Longrightarrow (ii)$, and  $(ii)\Longrightarrow (i)$ are obvious (since $T_k \in \mathbf{Sp}\{ R_1,\dots, R_n\}$ for $k=1, \dots, d).$ Therefore, it suffices to prove the implication   $(i)\Longrightarrow (iii).$

$(i)\Longrightarrow (iii).$ We assume that   $JtW_q(\mathbf{T})$ is convex. Let $S_1,\dots,S_m$ are in $\mathbf{Sp}\{T_1,\dots,T_d\},$ then for each $j\in \{ 1, \dots, m\},$ there exist $a_{j,1},  \dots, a_{j,d} \in \mathbb{C}$ such that the operator $S_j$ can be written as
\begin{equation}\label{expr2}
S_{j}=\sum_{k=1}^{d} a_{j,k} T_k.
\end{equation}
Now let   $(x_1, y_2), (x_2, y_2) \in \mathbb{S}_q(\mathcal{H})$ and  $t\in [0,1].$ It follows from  \eqref{expr2} that 
\begin{align}\label{ega3}
t \langle S_j x_1, y_1 \rangle + (1-t) \langle S_j x_2, y_2 \rangle &=\sum_{k=1}^{d}a_{j,k}\big[t \langle T_k x_1,y_1 \rangle +(1-t) \langle T_k x_2,y_2 \rangle\big]
\end{align}
for every $j\in \{1, \dots, m\}$. On the other hand, by the convexity of  $JtW_q(\mathbf{T}),$ there exists $(u_{t}, w_t) \in \mathbb{S}_q(\mathcal{H})$ such that 
\begin{align}\label{ega4}
t \langle T_k x_1,y_1 \rangle +(1-t) \langle T_k x_2,y_2 \rangle &= \langle T_k u_t, w_t \rangle \ \ (k\in \{ 1, \dots, d\}).
\end{align}
Hence, by combining \eqref{ega3} and \eqref{ega4}, we get that 
\begin{align*}
t \langle S_j x_1, y_1 \rangle + (1-t) \langle S_j x_2, y_2 \rangle &= \langle S_j u_t, w_t \rangle   \ \ (j\in \{ 1, \dots, m\}).
\end{align*}
From this we can conclude that the set  $JtW_q(S_1,\dots,S_m)$ is convex. This completes the proof.

\end{proof}

An immediate consequence of  Theorem  \ref{THMPRIN} is the following result.

\begin{corollary}
Let $T  \in \mathcal{B(H)}$. Then $JtW_q(a_1T+b_1 I,\dots,a_dT+b_dI)$ is convex, whenever   $a_i,b_i \in \mathbb{C}$ for all $i\in \{ 1, \dots, d\}$.
\end{corollary}

\begin{proof}
Clearly,  $a_i T +b_i I \in \mathbf{Sp}\{T, I\}$ for all $i\in \{ 1, \dots, d\}.$  On the other hand, it follows from Proposition \ref{prop1} (ii) that $JtW_q(T, I)$ is convex, since $W_q(T)$ is also convex.   Therefore, by applying Theorem \ref{THMPRIN}, we obtain the desired result.
\end{proof}

Another consequence of Theorem \ref{THMPRIN}  is the next result.

\begin{theorem}\label{THM5NEW}
Let $\mathbf{T}=(T_1,T_2,\cdots,T_d) \in \mathcal{B(H)}^d$, with $\dim \mathcal{H}=2.$ If there exists $i \in \{1, \dots, d \}$ such that $T_{i}$ is not a scalar operator and $T_{i}T_{j}=T_{j}T_{i}$ for all $j\in \{ 1, \dots, d \}$ with $j\neq i,$ then $JtW_q(\mathbf{T})$ is convex.
\end{theorem}

\begin{proof}
Without loss of generality, we may suppose that  $i=1.$  It follows from \cite[THEOREM II]{eliezer1968note} that  for each $j\in \{ 2, \dots, d \},$  there exsist $a_j, b_j \in \mathbb{C}$  such that  $T_j= a_j T_1 +b_j I.$ This implies that   $ T_1,\dots,T_d \in \mathbf{Sp}\{T_1,I\}.$
Therefore, by applying Theorem \ref{THMPRIN}, we achieve the required result.
\end{proof}

\begin{remark}
The example \ref{Example Intro} shows that  Theorem \ref{THM5NEW} breaks down when the commutativity assumption is not satisfied.  
\end{remark}

The next result is a direct consequence of Theorem  \ref{THM5NEW}.

\begin{corollary}
Let $\mathbf{T}=(T_1,\dots,T_d)\in\mathcal{B(H)}^d$ be a commuting $d$-tuple of operators, where $\mathcal{H}$ is a two-dimensional Hilbert space. Then,  $JtW_q(\mathbf{T})$ is convex.
\end{corollary}
\begin{remark}
Take two self-adjoint operators  $T_1=\begin{pmatrix}
	0 & m \\
	m & 0 
\end{pmatrix}$ and $T_2=\begin{pmatrix}
m &0 \\
0 & 0 
\end{pmatrix}$ on $\mathbb{R}^2$ acting on $(\mathbb{R}^2,\|.\|_2)$, where $m$ is non-zero real number. Then
\begin{equation*}
	JtW_q(T_1,T_2)=\{(mx_2y_1+mx_1y_2,mx_1y_1):~ x_1^2+x_2^2=1,~ y_1^2+y_2^2=1,~ x_1y_1+x_2y_2=q\}.
\end{equation*}
Here, $(m\sqrt{1-q^2},0) \in JtW_q(T_1,T_2) $ and $(-m\sqrt{1-q^2},mq) \in JtW_q(T_1,T_2)$ but $\frac{1}{2}(m\sqrt{1-q^2},0)+\frac{1}{2}(-m\sqrt{1-q^2},q)=\left(0, \frac{mq}{2}\right) \notin JtW_q(T_1,T_2)$. Thus, $JtW_q(T_1,T_2)$ is not convex. 
In this example, we observed that even if $T_i, 1 \leq i \leq d$ are self-adjoint operators, $JtW_q(\mathbf{T})$ is still not convex in general.
\end{remark}
We now focus to examine the relationships between the point spectrum $\sigma_p(\mathbf{T})$, the approximate point spectrum $\sigma_{ap}(\mathbf{T})$, and the joint $q$-numerical range $W_q(\mathbf{T})$ of $\mathbf{T}$.

Let  $\mathbf{T}=(T_{1}, \dots, T_{d}) \in \mathcal{B(H)}^d$ be a commuting $d$-tuple of operators that is $T_iT_j=T_jT_i$ for all $i,j \in \{1,2,...,d\}$. The point spectrum $\sigma_p(\mathbf{T})$ and the  approximate point spectrum $\sigma_{ap}(\mathbf{T})$ of $\mathbf{T}$ are respectively given by:
\begin{align*}
\sigma_p(\mathbf{T})&:=\left\{ (\xi_1,\dots,  \xi_d) \in \mathbb{C}^d :  \bigcap_{i=1}^{d} \mathcal{N}(T_i - \xi_i  I) \neq \{0 \} \right\}
\end{align*}
and 
\begin{align*}
\sigma_{ap}(\mathbf{T})&=\left\{ (\xi_1,\dots,  \xi_d) \in \mathbb{C}^d  :  \inf\left\{ \sum_{i=1}^{d} \Vert (T_i - \xi_i  I)x \Vert : x\in \mathcal{H}, \Vert x \Vert =1\right\}=0   \right\}.
\end{align*}

Note that in the case of $d$-tuple  $\mathbf{T}=(T_{1}, \dots, T_{d})$ of non-commuting operators, we cannot guarantee that these sets are non-empty. Indeed, if we take   $T_1=\begin{pmatrix}
	0 & 1 \\
	0 & 0 
\end{pmatrix}$ and $T_2=\begin{pmatrix}
1 &0 \\
0 & 0 
\end{pmatrix}$, then by a simple calculation, we can show that  $T_{1}T_{2} \neq T_{2} T_{1}$ and $\sigma_{p}(T_1,T_2)=\sigma_{ap}(T_1, T_2)=\emptyset.$

For more information concerning these sets and their applications, we refer the reader to  \cite[p. 91]{r.vm}.

\begin{theorem}
	Let $\mathbf{T}=(T_1,\dots,T_d)\in \mathcal{B(H)}^d.$  Then
	\begin{itemize}
		\item [(i)] $q \sigma_p(\mathbf{T}) \subseteq JtW_q(\mathbf{T})$,
		\item [(ii)] $q\sigma_{ap}(\mathbf{T})\subseteq \overline{JtW_q(\mathbf{T})},$
		where $\overline{JtW_q(\mathbf{T})}$ denotes the closure of $JtW_q(\mathbf{T})$.		
	\end{itemize}
	
\end{theorem}
\begin{proof}
	\begin{itemize}
\item[(i)]	 Let $\xi=(\xi_1,\dots,\xi_d) \in \sigma_p(\mathbf{T}).$ By the definition of $\sigma_p(\mathbf{T}),$ there exists a unit vector $x$ of $\mathcal{H}$  such that $T_{i}x=\xi_i x$  for all $i\in \{ 1, \dots, d \}.$ Let $z\in \mathcal{H}$ such that $\Vert z \Vert =1$ and $\langle x, z \rangle=0.$ We put $y= \overline{q}x+ \sqrt{1-|q|^2}z$. Clearly $(x,y)\in \mathbb{S}_q(\mathcal{H}).$  Moreover, for every $i\in \{ 1, \dots, d\},$ we have
	\begin{align*}
			\xi_i q= \xi_i \langle x,y \rangle
			= \langle \xi_i x,y \rangle
			= \langle T_ix,y \rangle.
		\end{align*}
	This implies that $\xi q \in 	JtW_q(\mathbf{T}).$ Consequently,  $q \sigma_p(\mathbf{T}) \subseteq JtW_q(\mathbf{T})$.

	\item [(ii)]
	Let $\xi=(\xi_1,\dots,\xi_d) \in \sigma_{ap}(\mathbf{T}),$ then there exists a sequence $\{x_m\}$ of unit vectors of $\mathcal{H}$ such that 
\begin{align*}
\lim_{m\to +\infty }\big\|( T_i-\xi_i I)x_m \big\| =0   \ \  (i=1, \dots, d).
\end{align*}	
On the other hand, 	for each $x_m \in \mathcal{H}$, there exists  $y_m \in \mathcal{H}$ such that $(x_m, y_m) \in \mathbb{S}_q(\mathcal{H})$. In particular,  $\langle Tx_m,y_m \rangle \in JtW_q(\mathbf{T}).$ Using Cauchy-Schwarz's inequality, we get that

	\begin{align*}
		|\langle T_ix_m,y_m \rangle -q\xi_i|=&|\langle T_ix_m,y_m \rangle -\xi_i\langle x_m,y_m \rangle|\\
		=&|\langle (T_i-\xi_i)x_m,y_m \rangle| \\
		\le& \|(T_i-\xi_i)x_m\| \|y_m\|\\
		=&\|(T_i-\xi_i)x_m\|,
	\end{align*}
for all $i\in \{1, \dots, d\}$ and all $m\in \mathbb{N}.$	This implies that 
\begin{align*}
\lim_{m\to +\infty} \langle T_ix_m,y_m \rangle = q \xi_i \ \  (i=1, \dots, d).
\end{align*}
Therefore $q\xi \in\overline{JtW_q(\mathbf{T})},$ and hence $q\sigma_{ap}(\mathbf{T})\subseteq \overline{JtW_q(\mathbf{T})}.$
		\end{itemize}
\end{proof}

Next, we aim to explore the relationship between the joint $C$-numerical range defined in \cite[p.371]{gau2021numerical}, where $C$ denotes a conjugation that is antilinear isometric involution; ($C \in \mathcal{B}(\mathcal{H})$~(\text{space of all bounded antilinear operators in the space} ~$\mathcal{H}$), $C^2=I$, ~\text{and} ~$\langle x,y \rangle=\langle Cx, Cy\rangle$ ~\text{for all} ~$x,y \in \mathcal{H})$ \cite{ptak2019c}, and the joint $q$-numerical range. Before proceeding, let us recall some fundamental facts.

Let $\mathcal{M}_n(\mathbb{C})$ denote the $C^*$-algebra of all $n \times n$ matrices. For $\mathbf{T} = (T_1, T_2, \ldots, T_d)$ with $T_i \in \mathcal{M}_n(\mathbb{C})$; $1 \le i \le d$ and a fixed matrix $C \in \mathcal{M}_n(\mathbb{C})$, the joint $C$-numerical range of $\mathbf{T}$ is defined as:
\begin{equation*}
    JtW_C(\mathbf{T}) = \left\{ \operatorname{tr}(CU^* \mathbf{T}U) : U \text{ is an } n \times n \text{ unitary matrix} \right\},
\end{equation*}
where $\operatorname{tr}(CU^* \mathbf{T}U) = (\operatorname{tr}(CU^* T_1 U), \operatorname{tr}(CU^* T_2 U), \ldots, \operatorname{tr}(CU^* T_d U))$. Here, for any $A \in \mathcal{M}_n(\mathbb{C})$, $tr(A)$ denotes the trace of $A$.  Clearly,
\begin{equation*}
    JtW_C(\mathbf{T}) = Jt_{V^* CV}(U^* \mathbf{T} U),
\end{equation*}
for any pair of $n \times n$ unitary matrices $U$ and $V$, where the $d$-tuple $U^* \mathbf{T} U$ is given by $U^* \mathbf{T} U := (U^* T_1 U, U^* T_2 U, \ldots, U^* T_d U)$.

In the following theorem, we will establish connections between the joint $C$-numerical range and the joint $q$-numerical range.

\begin{theorem}
Let $\mathbf{T}=(T_1,T_2,...,T_d) \in \mathcal{B}(\mathcal{H})^d$. Then
\begin{itemize}
	\item [(i)] If $\{T_i\}_{i=1}^d$ are bounded linear operators on a two-dimensional Hilbert space $\mathcal{H}$ and $|q|\le 1$, then $JtW_q(\mathbf{T})=JtW_C(\mathbf{T})$ for $C= 
	\begin{pmatrix}
		{q} & \sqrt{1-|q|^2}\\
		0 & 0
	\end{pmatrix}
	$,
\item [(ii)] If $\{T_i\}_{i=1}^d$ and $C$ are $n \times n ~(n \ge 2)$ matrices and $|q|\le 1$, then $JtW_q(\mathbf{T})=JtW_C(\mathbf{T})$, where $C=\begin{pmatrix}
{q} & \sqrt{1-|q|^2}\\
0 & 0
\end{pmatrix} \oplus 0_{n-2}$, 
\item [(iii)] If $C$ is $n \times n$ rank one matrix and $q \in [0,1]$, then $JtW_C(\mathbf{T})=\mu JtW_q(\mathbf{T})$, where $\mu$ is a non-zero scalar.
\end{itemize}
\end{theorem}
\begin{proof}
\begin{itemize}
	\item [(i)]

Let $\langle \mathbf{T}x,y \rangle \in JtW_{|q|}(\mathbf{T}) $, where $(x,y) \in \mathbb{S}_{|q|}(\mathbb{C}^2)$, thus 
$$y=|q|x+\sqrt{1-|q|^2}z$$ for some unit vector $z$ orthogonal to $x$. Then $T_i$ can be expressed as $\begin{pmatrix}
	\langle T_ix,x \rangle & 	\langle T_iz,x \rangle\\
		\langle T_ix,z \rangle & 	\langle T_iz,z \rangle
\end{pmatrix}$ with respect to the orthonormal basis $\{x,z\}$ of $\mathbb{C}^2$. Take $C'=\begin{pmatrix}
|q| & \sqrt{1-|q|^2}\\
0 & 0
\end{pmatrix}$. We have 
$$\langle T_ix,y \rangle=|q|	\langle T_ix,x \rangle+\sqrt{1-|q|^2}	\langle T_ix,z \rangle=tr(C'T_i).$$
 This implies, 
$\langle \mathbf{T}x,y \rangle=tr(C'\mathbf{T})=(tr(C'T_1),tr(C'T_2),...,tr(C'T_d))$. Thus 
$$JtW_{|q|}(\mathbf{T})\subseteq JtW_{C'}(\mathbf{T}).$$
 Conversely, Let $tr(C'U^*\mathbf{T}U) \in JtW_{C'}(\mathbf{T})$, where $U$ is $2 \times 2 $ unitary matrix. Letting $U^*T_iU=(a^{(i)}_{\ell m})_{\ell,m=1}^2$ $(1 \le i \le d)$, $x=U\begin{pmatrix}
1 \\ 0
\end{pmatrix}$ and $z=U\begin{pmatrix}
0 \\ 1
\end{pmatrix}$, we obtain that
\begin{align*}
tr(C'U^*A_iU)=&|q|a^{(i)}_{11}+\sqrt{1-|q|^2}a^{(i)}_{21}\\
=&|q| \Big \langle (U^* T_i U)\begin{pmatrix}
	1 \\ 0
\end{pmatrix}, \begin{pmatrix}
1 \\ 0
\end{pmatrix} \Big \rangle +\sqrt{1-|q|^2}\Big \langle (U^* T_i U)\begin{pmatrix}
1 \\ 0
\end{pmatrix}, \begin{pmatrix}
0 \\ 1
\end{pmatrix} \Big \rangle\\
=& |q| \langle T_i
x,x \rangle + \sqrt{1-|q|^2}\langle T_i
x,z \rangle\\
=&  \langle T_i
x,|q|x+\sqrt{1-|q|^2}z \rangle.
\end{align*}
If $y=|q|x+\sqrt{1-|q|^2}z$, then $(x,y)\in\mathbb{S}_q(\mathcal{H})$.
Therefore, $tr(C'U^* \mathbf{T}U)\in JtW_{|q|}(\mathbf{T})$ and $JtW_{C'}(\mathbf{T}) \subseteq JtW_{|q|}(\mathbf{T})$. Thus,  $JtW_{|q|}(\mathbf{T})= JtW_{C'}(\mathbf{T})$.
Take $q=e^{i \phi}|q|;~\phi$ is any real number, then by using Proposition 1 (v) and unitarily similarity of $e^{i\phi}C'$ and $C$, we have 
$$ JtW_q(\mathbf{T})= e^{i \phi}JtW_{|q|}(\mathbf{T})=e^{i \phi}JtW_{C'}(\mathbf{T})=JtW_{e^{i \phi}C'}(\mathbf{T})=JtW_{C}(\mathbf{T}).$$
 Hence, the result.
\item [(ii)] By using similar argument analogous in part (i), let $S=span\{x,z\}$, then $\mathbb{C}^n=S \oplus S^\perp$. Let $\{x,z,v_1,v_2,...v_{n-2}\}$ is an orthonormal basis for $\mathbb{C}^n$, where $\{v_1,v_2,...v_{n-2}\}$ is an orthonormal basis for $S^{\perp}$. Thus, $T_i ~(1 \le i \le d)$ can be expressed as $$\begin{pmatrix}
	\langle T_ix,x \rangle & 	\langle T_iz,x \rangle & 	\langle T_iv_1,x \rangle& \cdots & 	\langle T_iv_{n-2},x \rangle \\
	\langle T_ix,z \rangle & 	\langle T_iz,z \rangle & 	\langle T_iv_1,z \rangle& \cdots &  	\langle T_iv_{n-2},z \rangle \\
	\langle T_ix,v_1 \rangle & 	\langle T_iz,v_1 \rangle & 	\langle T_iv_1,v_1 \rangle& \cdots &  	\langle T_iv_{n-2},v_1 \rangle \\
	\vdots & \vdots & \vdots & \vdots & \vdots \\
	\langle T_ix,v_{n-2} \rangle & 	\langle T_iz,v_{n-2} \rangle & 	\langle T_iv_1,v_{n-2} \rangle& \cdots &  	\langle T_iv_{n-2},v_{n-2} \rangle \\
	
\end{pmatrix}$$ with respect to the orthonormal basis $\{x,z,v_1,v_2,...v_{n-2}\}$ of $\mathbb{C}^n$. Therefore, $JtW_{|q|}(\mathbf{T})\subseteq JtW_{C'}(\mathbf{T})$, where $C'=\begin{pmatrix}
|q| & \sqrt{1-|q|^2}\\
0 & 0
\end{pmatrix} \oplus 0_{n-2}$,  . Conversely, as similar to part (i), by taking $U^*T_iU=(a^{(i)}_{\ell m})_{\ell,m=1}^n$, $(1 \le i \le d)$, $x=U(1,0,...,0)^{T}$ and $z=U(0,1,0,...,0)^{T}$, we obtain that $JtW_{C'}(\mathbf{T})\subseteq JtW_{|q|}(\mathbf{T}) $, where $(1,0,...,0)^{T}$ and $(0,1,0,...,0)^{T}$ denote the transpose of $(1,0,...,0)$ and $(0,1,0,...,0)$ respectively. Thus $JtW_{|q|}(\mathbf{T})= JtW_{C'}(\mathbf{T})$ and $JtW_{q}(\mathbf{T})= JtW_{C}(\mathbf{T})$.
\item [(iii)] Since $C$ is a rank one matrix, $C$ is unitarily similar to a matrix of the form $\begin{pmatrix}
	u_1 & u_2 \\
	0 & 0
\end{pmatrix} \oplus 0_{n-2}$, where $u_1$ and $u_2$ are scalars. Matrix $\begin{pmatrix}
u_1 & u_2 \\
0 & 0
\end{pmatrix}$ is unitarily similar to $e^{it}\sqrt{|u_1|^2+|u_2|^2}\begin{pmatrix}
\frac{|u_1|}{\sqrt{|u_1|^2+|u_2|^2}} & \frac{|u_2|}{\sqrt{|u_1|^2+|u_2|^2}}\\
0 & 0
\end{pmatrix}$, where $t$ is any real number. We may take $C= \mu \begin{pmatrix}
q & \sqrt{1-q^2}\\
0 & 0 
\end{pmatrix}\oplus 0_{n-2}$. Thus, by part (ii), we have $JtW_C(\mathbf{T})= \mu JtW_q(\mathbf{T})$. 
\end{itemize}
\end{proof}

We now introduce the $q$-joint numerical radius $Jtw_q(\mathbf{T})$ as follows,
\begin{equation*}
	Jt\omega_q(\mathbf{T})=\sup\{\Vert(\langle T_1x,y\rangle,\dots,\langle T_d x,y\rangle)\Vert_2:(x,y) \in \mathbb{S}_{q}(\mathcal{H})\}.
\end{equation*}
The following proposition contains some properties for $q$-joint numerical radius.
\begin{proposition}
	Let $\mathbf{T}=(T_1,T_2,\cdots,T_d)\in \mathcal{B(H)}^d$ and $\mathbf{S}=(S_1,S_2,\cdots,S_d)\in \mathcal{B(H)}^d$ and $q$ be any complex number with $|q| \leq 1$. Then
	\begin{itemize}
		\item[(i)] $Jt\omega_q(\xi \mathbf{T})=\vert{\xi}\vert Jt\omega_q(\mathbf{T})$ for any complex number $\xi$,
		\item[(ii)] $Jt\omega_q(\mathbf{T}+\mathbf{S})\leq Jt\omega_q(\mathbf{T})+Jt\omega_q(\mathbf{S})$,
		\item[(iii)]$Jt\omega_{\overline{q}}(\mathbf{T}^*)=Jt\omega_q(\mathbf{T})$ and $Jt\omega_{q}(\mathbf{T}^*)=Jt\omega_{\overline{q}}(\mathbf{T})$,
		\item[(iv)] $Jt\omega_q(U^*\mathbf{T}U)=Jt\omega_q(\mathbf{T})$, where $U$ is a unitary operator on $\mathcal{H}$,		
		\item [(v)]$Jt\omega_q(\Re(\mathbf{T}))\le \frac{1}{2}\left( Jt\omega_q(\mathbf{T}) +Jt\omega_{\overline{q}}(\mathbf{T})\right)$ and $Jt\omega_q(\Im(\mathbf{T}))\le \frac{1}{2}\left( Jt\omega_q(\mathbf{T}) +Jt\omega_{\overline{q}}(\mathbf{T})\right)$, where $\Re(\mathbf{T})=(\Re(T_1),\Re(T_2),...,\Re(T_d))$ denotes the real part of $\mathbf{T}$ and $\Im(\mathbf{T})=(\Im(T_1),\Im(T_2),...,\Im(T_d))$ denotes the imaginary part of $\mathbf{T}$.
	\end{itemize}
\end{proposition}
\begin{proof}
Part (i) and (iii) follow from the definition directly and (ii) follows from the definition and Minkowski inequality. (iv) follows from Proposition \ref{prop1}(i). By using the Cartesian decomposition of $T_i$, $1 \le i \le d$ and Minkowski inequality, we can obtain (v) easily.
\end{proof}
\begin{theorem}
Let $|q| \le 1$. $Jt\omega_q(.)$ defines a semi-norm on $\mathcal{B}(\mathcal{H})^d$. Moreover, $Jt\omega_q(.)$ is a norm on $\mathcal{B}(\mathcal{H})^d$ if and only if $\omega_q(.)$ is a norm on $\mathcal{B}(\mathcal{H})$. 
\end{theorem}
\begin{proof}
It is straightforward to demonstrate that $Jt\omega_q(.)$ constitutes a semi-norm. Through basic calculations, it becomes evident that when $\omega_q(.)$ is a norm, then $Jt\omega_q(\cdot)$ also forms a norm. Now, let's proceed to establish the converse part for the latter assertion. Suppose $Jtw_q(.)$ forms a norm on $\mathcal{B(H)}^d$ but $\omega_q(.)$ does not form a norm on $\mathcal{B(H)}$. Then there exist a non-zero operator $S$ such that $\omega_q(S)=0$. Take $\mathbf{T}=(S,O,O,...,O)$, where $O$ denotes zero operator on $\mathcal{B(H)}$. Clearly, $\mathbf{T} \ne \mathbf{O}$; $\mathbf{O}=\{O,O,...,O\}$ and $Jt\omega_q(\mathbf{T})=0$ as $\omega_q(S)=0$. Thus $Jt\omega_q(.)$ is not a norm which is a contradiction.
\end{proof}

As we know that $\omega_q(T)$ defines a norm for all $q$ such that $|q| \leq 1$ excluding $q = 0$, we can state the following corollary.
\begin{corollary}
$Jt\omega_q(.)$ defines a norm on $\mathcal{B}(\mathcal{H})^d$ for all $q\in \mathcal{D}\setminus \{0\}$ $\left(see \text{\cite[Theorem 3.1]{li1994generalized}}\right)$, where $\mathcal{D}$ denotes the closed unit disk.
\end{corollary} 

In the following theorem, we provide a relation between $Jt\omega_q(\mathbf{T})$ and $\Vert{\mathbf{T}}\Vert$.
\begin{theorem}
	Let $\mathbf{T}=(T_1,T_2, \cdots,T_d) \in \mathcal{B(H)}^d$ and $q\in (0,1)$. Then
	\begin{equation}\label{ne1}
		\frac{q}{2\sqrt{d}(2-q^2)}\Vert{\mathbf{T}}\Vert\leq Jt\omega_q(\mathbf{T})\leq\Vert{\mathbf{T}}\Vert.
	\end{equation}
\end{theorem}

\begin{proof}
	Let $x,y\in \mathcal{H}$ and $\mathbf{T}=(T_1,T_2,\dots,T_d)\in \mathcal{B(H)}^d$, we have
	\begin{align*}
		\Vert(\langle{T_1x,y}\rangle, \cdots, \langle{T_dx,y}\rangle)\Vert_2 
		=\left(\sum\limits_{i=1}^{d}\vert\langle{T_ix,y}\rangle\vert^2\right)^\frac{1}{2} 
		 \le \left(\sum\limits_{i=1}^{d}\Vert{T_ix}\Vert^2\Vert{y}\Vert^2\right)^\frac{1}{2}.		
	\end{align*}
Taking supremum over all $(x,y)\in \mathbb{S}_q({\mathcal{H}})$, we obtain that $JtW_q(\mathbf{T}) \le \| \mathbf{T}\|.$

Now, by using inequality \eqref{eq1} and taking $x \in \mathbb{S}^1$, we have
	\begin{align*}
	 \left(\sum_{i=1}^d \|T_ix\|^2\right)^\frac{1}{2} 
		 \le\sqrt{d}\sup_{i=\{1,2,...,d\}}\Vert{T_ix}\Vert \le \sqrt{d}\sup_{i=\{1,2,...,d\}}\Vert{T_i}\Vert.
		\end{align*}
Taking supremum over $x\in \mathbb{S}^1$, we have
	\begin{align*}
\| \mathbf{T} \|  
&\le \sqrt{d}\sup_{i=\{1,2,...,d\}}\Vert{T_i}\Vert\\
&\leq\frac{2(2-q^2)}{q}\sqrt{d}\sup_{i=\{1,2,...,d\}}\omega_q(T_i),
	\end{align*}
where the last inequality follows from Theorem 2.1 \cite{fakhri2022q}. Therefore,
	\begin{align*} 
	\| \mathbf{T} \|& \le \frac{2\sqrt{d}(2-q^2)}{q}\sup_{i=\{1,2,...,d\}}\sup\limits_{(x,y)\in \mathbb{S}_q({\mathcal{H}})}\vert\langle{T_ix,y}\rangle\vert\\
		&\leq \frac{2\sqrt{d}(2-q^2)}{q}\sup\limits_{(x,y)\in \mathbb{S}_q({\mathcal{H}})}\sup_{i=\{1,2,...,d\}}\vert\langle{T_ix,y}\rangle\vert\\
		&=\frac{2\sqrt{d}(2-q^2)}{q}\sup\limits_{(x,y)\in \mathbb{S}_q({\mathcal{H}})}\|(\langle{T_1x,y}\rangle,\langle{T_2x,y}\rangle,...\langle{T_nx,y}\rangle)\|_{\infty} \\
		& \le \frac{2\sqrt{d}(2-q^2)}{q}\sup\limits_{(x,y)\in \mathbb{S}_q({\mathcal{H}})}\|(\langle{T_1x,y}\rangle,\langle{T_2x,y}\rangle,...\langle{T_nx,y}\rangle)\|_{2} \\
		&=\frac{2\sqrt{d}(2-q^2)}{q}Jt\omega_q(\mathbf{T}).
	\end{align*}
	Hence the required result holds.
\end{proof}
\begin{remark}
Here, it's important to note that there is an error in the calculation presented in Theorem 2.1 \cite{fakhri2022q}. Upon correcting this, if we proceed with the accurate calculations, Theorem 2.1 \cite{fakhri2022q} gives the following inequality 
\begin{equation*}
	\frac{q}{2}\|T\| \le w_q(T) \le \|T\|.
\end{equation*}
Accordingly, \eqref{ne1} reduces to
\begin{equation}\label{nb}
	\frac{q}{2\sqrt{d}}\Vert{\mathbf{T}}\Vert\leq Jt\omega(\mathbf{T})\leq\Vert{\mathbf{T}}\Vert.
\end{equation}
When we put $q=1$ in inequalities \eqref{ne1} and \eqref{nb}, we obtain the following inequality
\begin{center}
	$\frac{1}{2\sqrt{d}}\Vert{\mathbf{T}}\Vert\leq Jt\omega(\mathbf{T})\leq\Vert{\mathbf{T}}\Vert$,
\end{center}
 which is highlighted in Theorem 9 \cite{drnovvsek2014joint}.
\end{remark}

In the next theorem, we obtain an upper bound for $q$-joint numerical radius where $T_i's$ are $2 \times 2$ operator matrices. 
\begin{theorem}
Let $\mathbf{T}=(T_1,T_2,...T_d)$ and $T_i=\begin{pmatrix}
	P_i & Q_i \\
	R_i & S_i
\end{pmatrix}$, where $P_i,Q_i,R_i,S_i \in \mathcal{B(H)}$, $1 \le i \le d$. Then we have
\begin{align*}
&\max\{Jt\omega_q^2(\mathbf{P}),Jt\omega_q^2(\mathbf{S})\}\\
&\le 	Jt\omega_q^2(\mathbf{T})\\
&\le \sum_{i=1}^d \Big(\frac{|q|}{2}\left(\omega(P_i)+\omega(S_i)+\sqrt{\left(\omega(P_i)-\omega(S_i)\right)^2+(\|R_i\|+\|Q_i\|)^2}\right)\\
&	+\sqrt{1-|q|^2}\left(\|P_i\|^2+\|Q_i\|^2+\|R_i\|^2+\|S_i\|^2\right)^{\frac{1}{2}}
	\Big)^2,
\end{align*}
where $\mathbf{P}=(P_1,P_2,...,P_d)$, $\mathbf{S}=(S_1,S_2,...,S_d)$ and $|q|\le 1$.
\end{theorem}
\begin{proof}
Let $\mathbf{x}=\begin{pmatrix}
		x_1\\
		x_2
	\end{pmatrix}$, $\mathbf{y}=\begin{pmatrix}
		y_1\\
		y_2
	\end{pmatrix} \in \mathcal{S}_q (\mathcal{H} \oplus \mathcal{H}).$ Thus
\begin{equation*}
	Jt\omega_q^2(\mathbf{T})=\sup_{(\mathbf{x},\mathbf{y})\in \mathcal{S}_q(\mathcal{H} \oplus \mathcal{H})}\sum_{i=1}^d|\langle T_i \mathbf{x}, \mathbf{y}\rangle|^2.
\end{equation*}
In particular, let
$\mathbf{x}=\begin{pmatrix}
		x_1\\
		0
	\end{pmatrix}$, $\mathbf{y}=\begin{pmatrix}
		y_1\\
		0
	\end{pmatrix} \in \mathcal{H} \oplus \mathcal{H}$, where $\|\mathbf{x}\|=\|\mathbf{y}\|=1$ and $\langle \mathbf{x}, \mathbf{y} \rangle=x_1,y_1=q$.
Therefore,
\begin{align*}
	Jt\omega_q^2(\mathbf{T}) \ge &\sum_{i=1}^{d}\left| \big< \begin{pmatrix}
		P_i & Q_i \\
		R_i & S_i
	\end{pmatrix}\begin{pmatrix}
		x_1\\
		0
	\end{pmatrix},\begin{pmatrix}
		y_1\\
		0
	\end{pmatrix}\big> \right|^2\\
=&\sum_{i=1}^{d}\left|\langle P_i x_1, y_1 \rangle \right |^2.
\end{align*}
Taking supremum over all $x_1$ and $y_1$ such that $(x_1,y_1) \in \mathcal{S}_q$, we have
\begin{equation*}
		Jt\omega_q^2(\mathbf{T}) \le Jt\omega_q^2(\mathbf{P}).
\end{equation*}
Similarly, it can be obtained that
$	Jt\omega_q^2(\mathbf{T}) \ge Jt\omega_q^2(\mathbf{S}) $.
Thus, $$\max\{Jt\omega_q^2(\mathbf{P}),Jt\omega_q^2(\mathbf{S})\}\le	Jt\omega_q^2(\mathbf{T}).$$
Now, let $\mathbf{x}, \mathbf{y} \in \mathcal{S}_q$, we have
\begin{align*}
|\langle T_i \mathbf{x}, \mathbf{y}\rangle|=&\left| \big< \begin{pmatrix}
	P_i & Q_i \\
	R_i & S_i
\end{pmatrix}\begin{pmatrix}
x_1\\
x_2
\end{pmatrix},\begin{pmatrix}
y_1\\
y_2
\end{pmatrix}\big> \right|\\
\le & |\langle P_ix_1,y_1 \rangle|+|\langle Q_ix_2,y_1 \rangle|+|\langle R_ix_1,y_2 \rangle|+|\langle S_ix_2,y_2 \rangle|.
\end{align*}
This implies,
\begin{align*}
Jt\omega_q^2(\mathbf{T})=&\sup_{(\mathbf{x},\mathbf{y})\in \mathcal{S}_q (\mathcal{H} \oplus \mathcal{H})}\sum_{i=1}^d|\langle T_i \mathbf{x}, \mathbf{y}\rangle|^2\\
\le & \sup_{(\mathbf{x},\mathbf{y})\in \mathcal{S}_q (\mathcal{H} \oplus \mathcal{H})}\sum_{i=1}^d \left( |\langle P_ix_1,y_1 \rangle|+|\langle Q_ix_2,y_1 \rangle|+|\langle R_ix_1,y_2 \rangle|+|\langle S_ix_2,y_2 \rangle|\right)^2\\
 \le & \sum_{i=1}^d \left( \sup_{(\mathbf{x},\mathbf{y})\in \mathcal{S}_q (\mathcal{H} \oplus \mathcal{H})}\left(|\langle P_ix_1,y_1 \rangle|+|\langle Q_ix_2,y_1 \rangle|+|\langle R_ix_1,y_2 \rangle|+|\langle S_ix_2,y_2 \rangle|\right)\right)^2.\\
\end{align*}
Take $\mathbf{y}=\overline{q}\mathbf{x}+\sqrt{1-|q|^2}\mathbf{z}$,  where $\mathbf{z}=\begin{pmatrix}
	z_1\\
	z_2
\end{pmatrix}$, $\|\mathbf{z}\|=1$, $\langle \mathbf{x}, \mathbf{z} \rangle=0$. Thus, $y_1=\overline{q}x_1+\sqrt{1-|q|^2}z_1$ and $y_2=\overline{q}x_2+\sqrt{1-|q|^2}z_2$.
Suppose 
\begin{equation*}
\mu_i =|\langle P_ix_1,y_1 \rangle|+|\langle Q_ix_2,y_1 \rangle|+|\langle R_ix_1,y_2 \rangle|+|\langle S_ix_2,y_2 \rangle|.
\end{equation*}
Thus,
\begin{align*}
 \mu_i = & |\langle P_ix_1,\overline{q}x_1+\sqrt{1-|q|^2}z_1 \rangle|
+|\langle Q_ix_2,\overline{q}x_1+\sqrt{1-|q|^2}z_1 \rangle|\\
+&|\langle R_ix_1,\overline{q}x_2+\sqrt{1-|q|^2}z_2 \rangle|+|\langle S_ix_2,\overline{q}x_2+\sqrt{1-|q|^2}z_2 \rangle|\\
\le & |q|\left(|\langle P_i x_1, x_1 \rangle|+|\langle Q_i x_2, x_1 \rangle|+|\langle R_i x_1, x_2 \rangle|+|\langle S_i x_2, x_2 \rangle|\right)\\
+& \sqrt{1-|q|^2}\left(|\langle P_i x_1, z_1 \rangle|+|\langle Q_i x_2, z_1 \rangle|+|\langle R_i x_1, z_2 \rangle|+|\langle S_i x_2, z_2 \rangle|\right).
\end{align*}
By using definition of $\omega(T)$ and Cauchy-Schwarz inequality, we have
\begin{align*}
\mu_i \le &	|q|\left(\omega(P_i)\|x_1\|^2+\omega(S_i)\|x_2\|^2+(\|Q_i\|+\|R_i\|)\|x_1\|\|x_2\|\right)\\
+& \sqrt{1-|q|^2}\left(\|P_i\| \|x_1\|\| z_1\| +\|Q_i\| \|x_2\| \|z_1\| +\| R_i\| \|x_1\| \|z_2\| +\|S_i\| \|x_2\|\| z_2\|\right).
\end{align*}
 Putting $\|x_1\|=\cos\theta$, $\|x_2\|=\sin\theta$, $\|z_1\|=\cos\phi$ and $\|z_2\|=\sin\phi$, we obtain that
 \begin{align*}
 \sup_{(\mathbf{x},\mathbf{y})\in \mathcal{S}_q}\mu_i \le & |q|\left(\omega(P_i)\cos^2\theta+\omega(S_i)\sin^2\theta+(\|Q_i\|+\|R_i\|)\cos\theta\sin\theta \right)\\
 +& \sqrt{1-|q|^2}(\|P_i\| \cos\theta\cos\phi +\|Q_i\|\sin\theta\cos\phi \\
 +&\| R_i\| \cos\theta\sin\phi +\|S_i\| \sin\theta\sin\phi)\\
 \le &\frac{|q|}{2}\left(\omega(P_i)+\omega(S_i)+\sqrt{(\omega(P_i)-\omega(S_i))^2+(\|R_i\|+\|Q_i\|)^2}\right)\\
 +&\sqrt{1-|q|^2}\left(\|P_i\|^2+\|Q_i\|^2+\|R_i\|^2+\|S_i\|^2\right)^{\frac{1}{2}}.
 \end{align*}
Therefore,
\begin{align*}
Jt\omega_q^2(\mathbf{T})
\le&\sum_{i=1}^d \Big(\frac{|q|}{2}\left(\omega(P_i)+\omega(S_i)+\sqrt{\left(\omega(P_i)-\omega(S_i)\right)^2+\left(\|R_i\|+\|Q_i\|\right)^2}\right)\\
+&\sqrt{1-|q|^2}\left(\|P_i\|^2+\|Q_i\|^2+\|R_i\|^2+\|S_i\|^2\right)^{\frac{1}{2}}
\Big)^2.
\end{align*}
Hence, the required result holds.
\end{proof}

On the basis of above theorem, we can obtain the following remarks easily.
\begin{remark}
\begin{itemize}
Let $P,Q,R,S \in \mathcal{B(H)}$. Then, for all $q \in \mathbb{C}$ such that $|q|\le 1$, we have 
\begin{align*}
\max\{\omega_q({P}),\omega_q({Q})\} \le& 	\omega_q\begin{pmatrix}
	P & Q \\
	R & S
\end{pmatrix}\\
 \le & \frac{|q|}{2}\left(\omega(P)+\omega(S)+\sqrt{(\omega(P)-\omega(S))^2+\left(\|R\|+\|Q\|\right)^2}\right)\\
	+&\sqrt{1-|q|^2}\left(\|P\|^2+\|Q\|^2+\|R\|^2+\|S\|^2
	\right)^{\frac{1}{2}}.
\end{align*}
\item [(ii)] Let ${T}_i=\begin{pmatrix}
	P_i & Q_i \\
	0 & 0
\end{pmatrix}$ be such that $P_i,Q_i \in \mathcal{B(H)}$ for all $1 \le i \le d$. Then, for all $q \in \mathbb{C}$ such that $|q|\le 1$, we have 
\begin{align*}
	Jt\omega_q^2(\mathbf{P}) \le &	Jt\omega_q^2(\mathbf{T})\\
	\le& \sum_{i=1}^d \left(\frac{|q|}{2} \left(\omega(P_i)+\sqrt{\omega^2(P_i)+\|Q_i\|^2}\right)
	+\sqrt{1-|q|^2}\left(\|P_i\|^2+\|Q_i\|^2\right)^{\frac{1}{2}}
	\right)^2.
\end{align*}
In particular, if $P,Q \in \mathcal{B(H)}$ and ${T}=\begin{pmatrix}
	P& Q \\
	0 & 0
\end{pmatrix}$, then
\begin{equation}\label{ri}
	\omega_q(P) \le \omega_q({T}) \le  \frac{|q|}{2} \left(\omega(P)+\sqrt{\omega^2(P)+\|Q\|^2}\right)
	+\sqrt{1-|q|^2}\left(\|P\|^2+\|Q\|^2\right)^{\frac{1}{2}}.
\end{equation}
Right inequality in \eqref{ri} generalizes and refines the following inequality
\begin{equation*}
	\omega\begin{pmatrix}
		P & Q \\
		0 & 0
	\end{pmatrix} \le \omega(P) + \frac{\|Q\|}{2};~ P,Q \in \mathcal{B(H)},
\end{equation*} 
which is obtained by Omar Hirzallah, Fuad Kittaneh and
Khalid Shebrawi in \cite{hirzallah2011numerical}.
\item [(iii)] Let ${T}_i=\begin{pmatrix}
	P_i & 0 \\
	0 & S_i
\end{pmatrix}$ be such that $P_i,S_i \in \mathcal{B(H)}$ for all $1 \le i \le d$. Then, for all $q \in \mathbb{C}$ such that $|q|\le 1$, we have
\begin{align*}
\max\{Jt\omega_q^2(\mathbf{P}),Jt\omega_q^2(\mathbf{S})\} \le &	Jt\omega_q^2(\mathbf{T})\\
 \le &\sum_{i=1}^d \left(|q|\max \{\omega(P_i), \omega(Q_i) \}
	+\sqrt{1-|q|^2}\left(\|P_i\|^2+\|S_i\|^2\right)^{\frac{1}{2}}
	\right)^2.
\end{align*}
In particular, if $P,S \in \mathcal{B(H)}$ and ${T}=\begin{pmatrix}
	P & 0 \\
	0 & S
\end{pmatrix}$, then
\begin{equation}\label{la}
\max \{\omega_q({P}),\omega_q({S})\} \le 	\omega_q({T}) \le  |q|\max \{\omega(P), \omega(Q) \}
	+\sqrt{1-|q|^2}\left(\|P\|^2+\|S\|^2\right)^{\frac{1}{2}}.
\end{equation}
For $q=1$, \eqref{la} gives us the following well-known inequality
\begin{equation*}
	\omega\begin{pmatrix}
		P & 0\\
		0 & S
	\end{pmatrix}= \max\{\omega(P),\omega(Q)\}.
\end{equation*}
\item [(iv)]  Let ${T}_i=\begin{pmatrix}
	0 & Q_i \\
	R_i & 0
\end{pmatrix}$ be such that $Q_i,R_i \in \mathcal{B(H)}$ for all $1 \le i \le d$. Then, for all $q \in \mathbb{C}$ such that $|q|\le 1$, we have
\begin{align*}
	Jt\omega_q^2(\mathbf{T}) \le &\sum_{i=1}^d \left(\frac{|q|}{2}\left(\|R_i\|+\|Q_i\|\right)
	+\sqrt{1-|q|^2}\left(\|Q_i\|^2+\|R_i\|^2\right)^{\frac{1}{2}}
	\right)^2.
\end{align*}
 In particular, if $Q,R \in \mathcal{B(H)}$ and ${T}=\begin{pmatrix}
	0 & Q \\
	R & 0
\end{pmatrix}$, then
\begin{align*}
	\omega_q({T})\le & \frac{|q|}{2}
	\left(\|R\|+\|Q\|\right)
	+\sqrt{1-|q|^2}\left(\|Q\|^2+\|R\|^2\right)^{\frac{1}{2}}
	.
\end{align*}
\item [(v)]  Let ${T}_i=\begin{pmatrix}
	\pm P_i & \pm Q_i \\
	\pm Q_i & 	\pm P_i
\end{pmatrix}$ be such that $P_i,Q_i \in \mathcal{B(H)}$ for all $1 \le i \le d$. Then, for all $q \in \mathbb{C}$ such that $|q|\le 1$, we have
\begin{align*}
	Jt \omega ^2(\mathbf{P}) \le Jt\omega_q^2(\mathbf{T}) \le &\sum_{i=1}^d \left(|q|\left(\omega(P_i)+\|Q_i\|\right)
	+2\sqrt{1-|q|^2}\left(\|P_i\|^2+\|Q_i\|^2\right)^{\frac{1}{2}}
	\right)^2.
\end{align*}
In particular, If $P,Q \in \mathcal{B(H)}$ and ${T}=\begin{pmatrix}
	\pm P & \pm Q \\
	\pm Q & 	\pm P
\end{pmatrix}$, then
\begin{align*}
	\omega_q (P) \le \omega_q(T) \le & |q|\left(\omega(P)+\|Q\|\right)
	+2\sqrt{1-|q|^2}\left( \|P\|^2+\|Q\|^2\right)^{\frac{1}{2}}
	.
\end{align*}
\end{itemize}
\end{remark}

\section{$A$-$q$-Numerical Range and Convexity}
In this section, our objective is to extend and generalize the concept of the $q$-numerical range to the context of semi-Hilbert spaces. However, before delving into that, it is necessary to revisit some fundamental aspects of semi-Hilbert spaces. From a positive operator $A \in \mathcal{B(H)}$, we can obtain the following positive semi-definite sesquilinear form on $\mathcal{H}$:
\begin{equation*}
\langle x, y \rangle_{A}=\langle Ax, y \rangle \ \  \textup{ for all } x,y\in \mathcal{H}.
\end{equation*}
We denote by $\Vert \cdot \Vert_{A}$ the seminorm induced by $\langle \cdot, \cdot \rangle_{A}$, i.e., $\Vert x \Vert_{A}=\sqrt{\langle x, x \rangle_{A}}$ for all $x\in \mathcal{H}$.  A straightforward calculation shows that   $\Vert \cdot \Vert_{A}$  is a norm if and only if $A$ is injective, and that $\Big( \mathcal{H}, \Vert \cdot \Vert_{A} \Big)$ is a complete space if and only if  the range $\mathcal{R}(A)$ of $A$ is closed in $\mathcal{H}$. Further, we can easily check that   $\langle \cdot, \cdot \rangle_{A}$ satisfies the   Cauchy-Schwarz inequality:
\begin{equation*}
\vert \langle x, y \rangle_{A} \vert \leq \Vert x \Vert_{A} \Vert y \Vert_{A} \ \ (x,y \in \mathcal{H}).
\end{equation*}

Let $T\in \mathcal{B(H)}$,  an operator $R\in \mathcal{B(H)}$ is said to be $A$-adjoint of $T$ if for every $x,y\in \mathcal{H}$, we have
$\langle Tx,y \rangle_{A}=\langle x,Ry\rangle_{A}.$ In other words, $R$ is an $A$-adjoint of $T$ if and only if  $AR=T^{*}A$. This kind of equation is studied by Douglas in \cite{r.dog}. In general, the existence and uniqueness of $A$-adjoint are not guaranteed;  see  \cite{r.kf6}. We denote by $\mathcal{B}_{A}(\mathcal{H})$ the set of all operators in $\mathcal{B(H)}$ which admit  $A$-adjoints. From Douglas Theorem \cite{r.dog}, we have
\begin{equation*}
\mathcal{B}_{A}\mathcal{(H)}=\{ T \in \mathcal{B(H)} : \mathcal{R}(T^{*}A)\subseteq \mathcal{R}(A) \}.
\end{equation*}

Recall that an operator $T\in \mathcal{B(H)}$  is said to be $A$-bounded if there exists a positive real number $\xi$ such that $\Vert Tx \Vert_{A} \leq \xi \Vert x \Vert_{A}$ for all $x\in \mathcal{H}$. Applying Douglas's theorem \cite{r.dog} again, we find that the set of all $A$-bounded operators in $\mathcal{B(\mathcal{H})}$ is exactly $\mathcal{B}_{A^{\frac{1}{2}}}(\mathcal{H})$, in other words,
\begin{equation*}
\mathcal{B}_{A^{\frac{1}{2}}}(\mathcal{H})=\{ T \in \mathcal{B(H)} : \exists \xi >0,  \Vert Tx \Vert_{A} \leq \xi \Vert x \Vert_{A}, \forall x\in \mathcal{H} \}.
\end{equation*}

 Notice that the spaces $\mathcal{B}_{A}\mathcal{(H)}$ and  $\mathcal{B}_{A^{\frac{1}{2}}}(\mathcal{H})$  are two  subalgebras of $\mathcal{B(H)}$ that are  neither closed nor dense in $\mathcal{B(H)}$, see \cite[p.1463]{r.lacg1}.  Furthermore, we have the following chain of inclusions
 \begin{equation*}
  \mathcal{B}_{A}\mathcal{(H)} \subseteq \mathcal{B}_{A^{\frac{1}{2}}}(\mathcal{H}) \subseteq \mathcal{B(H)},
 \end{equation*}
with equality if $\mathcal{N}(A)=\{ 0 \}$ and $\overline{\mathcal{R}(A)}=\mathcal{R}(A)$, where $\mathcal{N}(A)$ is the kernel  of $A$.

 If $T\in \mathcal{B}_{A}(\mathcal{H})$, the reduced solution of the equation $AX = T^{*}A$ is a distinguished $A$-adjoint operator of $T$, which is denoted by $T^{\sharp_{A}}$ and satisfies $\mathcal{R}(T^{\sharp_{A}}) \subseteq \overline{\mathcal{R}(A)}$. Note that, $T^{\sharp_{A}} = A^{\dagger} T^{*}A$,  where  $A^{\dagger}$ is the Moore-Penrose inverse of $A$; see \cite{r.lacg}.  It is important to note that, if $T\in \mathcal{B}_{A}(\mathcal{H})$ then $T^{\sharp_{A}} \in \mathcal{B}_{A}\mathcal{(H)}$,  $\big(\big(T^{\sharp_{A}}\big)^{\sharp_{A}} \big)^{\sharp_{A}}=T^{\sharp_{A}}$ and  $\big(T^{\sharp_{A}}\big)^{\sharp_{A}}=PTP$, where $P$ represents  the orthogonal projection onto $\overline{\mathcal{R}(A)}$. Moreover,  we have $\big(TS\big)^{\sharp_{A}}=S^{\sharp_{A}} T^{\sharp_{A}}$ and $\big(T+\alpha S\big)^{\sharp_{A}}=T^{\sharp_{A}}+ \overline{\alpha} S^{\sharp_{A}}$ for all $T, S \in \mathcal{B}_{A}(\mathcal{H})$ and all $\alpha\in \mathbb{C}$, where $\overline{\alpha}$ is the conjugate of the complex number $\alpha$. If no confusion arises, we  simply write  $T^{\sharp}$ instead of $T^{\sharp_{A}}$.

In the sequel, $q$ is a complex number such that $\vert q \vert \leq 1.$ For a subspace $\mathcal{V}$ of $\mathcal{H}$, we denote by $\mathbb{S}_{q, A}(\mathcal{V})$ the subset of $\mathcal{V} \times \mathcal{V}$ given by
\begin{equation*}
\mathbb{S}_{q, A}(\mathcal{V})=\left\{ (x,y)\in \mathcal{V}\times \mathcal{V} : \Vert x \Vert_{A}=\Vert y \Vert_{A}=1 \textup{ and } \langle x, y \rangle_{A}=q \right\}.
\end{equation*}

\begin{definition}
Let $T\in \mathcal{B(H)}$.  The $q$-$A$-numerical range of $T$ is the subset of $\mathbb{C}$ defined by
\begin{equation*}
W_{q,A}(T)=\left \{ \langle Tx, y \rangle_{A} : (x,y) \in \mathbb{S}_{q}(\mathcal{H}) \right\}.
\end{equation*}

\end{definition}

\begin{proposition}
Let $T\in \mathcal{B(H)}$. Then, we have the following assertions.
\begin{enumerate}
\item $W_{e^{i \theta} q,A}(T)=e^{i\theta}W_{q , A}(T)$ for all $\theta \in \mathbb{R}$. In particular, $W_{ q,A}(T)=W_{\vert q \vert , A}(T)$.
\item $W_{1, A}(T)= W_{A}(T)$.
\item If $\dim \mathcal{R}(A)=1$, then  $W_{q,A}(T)\neq  \emptyset$  if and only if $\vert q \vert=1$.
\item If $T \in \mathcal{B}_{A}(\mathcal{H})$, then  $W_{q, A}(T^{\sharp})=\{ \overline{\xi} : \xi \in W_{q,A}(T) \}$.
\end{enumerate}

\end{proposition}

\begin{proof}
We will only show statement $(3)$, since the others are obvious.

 The `if' follows directly from the statement $(2)$. For the  `only if',  suppose that there exists a complex number  $\xi$ which belongs to $W_{q,A}(T)$.  Then there  exists $(x,y) \in \mathbb{S}_{q, A}(\mathcal{H})$ such that $\xi=\langle Tx, y \rangle_{A}$. Since  $\dim \mathcal{R}(A)=1$, there exists  $u \in \mathcal{H}$ such that $\Vert u \Vert=1$ and  $\mathcal{R}(A)$ equal to  the subspace spanned by $u$, moreover, we have that $\mathcal{R}(A)=\mathcal{R}(A^{\frac{1}{2}})$. This implies that there exists $\alpha , \beta \in \mathbb{C}$ such that $A^{\frac{1}{2}}x=\alpha u$ and $A^{\frac{1}{2}}y=\beta u$. Hence we get that   $\vert \alpha \vert = \vert \beta \vert= 1$ and $\vert \alpha \beta\vert   = \vert q \vert=1$. This completes the proof.
\end{proof}
We require the following Lemma in order to demonstrate the convexity of $W_{q,A}(T)$, which is the primary result of this section.

\begin{lemma}[\cite{r.wagschal}, p.444]\label{lem2}
Let $\left(E, \varphi(\cdot, \cdot) \right)$ be a pre-Hilbertian space with $\varphi(\cdot, \cdot)$ is the inner product on $E$. For any finite-dimensional subspace  $F$ of $E$, we have
\begin{enumerate}
\item $E=F\oplus F^{\bot_{\varphi}}$, where $F^{\bot_{\varphi}}$ is the  orthogonal complement of $F$ with respect the inner product $\varphi(\cdot, \cdot)$, i.e.  $F^{\bot_{\varphi}}=\{ x\in E: \varphi(x, y)=0, \forall y\in F\}$;
\item the mapping $\mathcal{Q}_{F} : E \rightarrow F$ given by  $\mathcal{Q}_{F}\left(x_{F}+x_{F^{\bot_{\varphi}}}\right)=x_{F}$  is the  orthogonal projection of $E$ on $F$.
\end{enumerate}
\end{lemma}

An immediate consequence of Lemma \ref{lem2} is the following result.

\begin{corollary}\label{corl4}
Let $\left(E, \varphi(\cdot, \cdot) \right)$ be a pre-Hilbertian space with $\varphi(\cdot, \cdot)$ is the inner product on $E$. If the dimension of $E$ is strictly greater than $1$, then  for every $x\in E$, there exists $y\in E\setminus \{0\}$ such that $\varphi(x, y)=0$.
\end{corollary}

In what follows, we assume that  $\dim \mathcal{R}(A)\geq 2$.
\begin{theorem}\label{thm21prin}
Let $T\in \mathcal{B}(\mathcal{H})$. Then $W_{q,A}(T)$ is convex.
\end{theorem}

\begin{proof}
Without loss of generality, we may assume that $0 \leq q\leq 1$. First of all, it should be noted that each vectors $x, y\in \mathcal{H}$ can be uniquely written as $x=x_1+x_2$ and $y=y_1+y_2$, where $x_1, y_1\in \mathcal{N}(A)$ and $x_2, y_2 \in \overline{\mathcal{R}(A)}$, since   $\mathcal{H}=\mathcal{N}(A) \oplus \overline{\mathcal{R}(A)}$.   By using this decomposition and the fact that $\mathcal{N}(A)=\mathcal{N}(A^{\frac{1}{2}})$, we can conclude that $(x,y) \in \mathbb{S}_{q, A}(\mathcal{H})$ if and only if $(x_2, y_2) \in \mathbb{S}_{q, A}\left(\overline{\mathcal{R}(A)}\right)$, moreover
\begin{align*}
&W_{q,A}(T)=\left \{ \langle Tx_1+Tx_2, y_2 \rangle_A : x_{1}\in \mathcal{N}(A), (x_2, y_2) \in \mathbb{S}_{q, A}\left(\overline{\mathcal{R}(A)}\right)\right\}.
\end{align*}
From this we have two possible cases:
\begin{enumerate}
\item[\textbf{Case 1:}] If $T(\mathcal{N}(A))\subseteq \mathcal{N}(A)$. In this case, we have
\begin{align}\label{equ6}
&W_{q,A}(T)=\left\{ \langle PTx_2, y_2 \rangle_{A}: (x_2,y_2)  \in \mathbb{S}_{q, A}\left(\overline{\mathcal{R}(A)}\right) \right\} =W_{q,A^{\prime}}(T^{\prime}),
\end{align}
where $A^{\prime}=A\mid_{\overline{\mathcal{R}(A)}}$ and $T^{\prime}=PT\mid_{\overline{\mathcal{R}(A)}}$. From the  equality \eqref{equ6}, it suffices to show that $W_{q,A^{\prime}}(T^{\prime})$ is convex. Before we begin our demonstration, it is important to note that  the space $\left(\overline{\mathcal{R}(A)}, \langle \cdot, \cdot \rangle_{A^{\prime}}\right)$ is pre-Hilbertien, since $A^{\prime} \in \mathcal{B}\left(\overline{\mathcal{R}(A)}\right)^{+}$ is injective. Now, let $x_{2}, y_{2}, f_{2}, g_{2} \in \overline{\mathcal{R}(A)}$ such that  $\Vert x_{2} \Vert_{A^{\prime}}=\Vert y_{2} \Vert_{A^{\prime}} = \Vert f_{2} \Vert_{A^{\prime}}=\Vert g_{2} \Vert_{A^{\prime}} =1$ and $\langle x_{2}, y_{2}\rangle_{A^{\prime}}= \langle f_{2}, g_{2} \rangle_{A^{\prime}}=q$. We will show that  the segment  joining  $\langle T^{\prime}x_{2}, y_{2}\rangle_{A^{\prime}}$ and  $\langle T^{\prime}f_{2}, g_{2} \rangle_{A^{\prime}}$ is contained in $W_{q,A^{\prime}}(T^{\prime})$. For this, let $F$ be the subspace spanned by the family $\{ x_{2}, y_{2}, f_{2}, g_{2}\}$. By virtue of  Lemma \ref{lem2}, there exists an orthogonal projection  $\mathcal{Q}_{F}$ of $\overline{\mathcal{R}(A)}$ on $F$ relative to   $\langle \cdot, \cdot \rangle_{A^{\prime}}$, since   the space $F$ has finite dimension. A simple calculation shows that the operator $\mathbf{T}''=\mathcal{Q}_{F}T^{\prime} \mathcal{Q}_{F}$ is bounded on $F$, moreover  $\langle \mathbf{T}''x_{2}, y_{2}\rangle_{A^{\prime}}=\langle T^{\prime}x_{2}, y_{2}\rangle_{A^{\prime}}$ and $\langle \mathbf{T}'' f_{2}, g_{2}\rangle_{A^{\prime}}=\langle T^{\prime}f_{2}, g_{2}\rangle_{A^{\prime}}$. We can also easily verify the following inclusion  $$W_{q, A^{\prime}}(\mathbf{T}'' )\subseteq W_{q, A^{\prime}}(T^{\prime}),$$ where
$$W_{q, A^{\prime}}(\mathbf{T}'')=\left\{ \langle \mathbf{T}''u, v\rangle_{A^{\prime}} : (u,v) \in \mathbb{S}_{q, A^{\prime}}(F) \right\}.$$
Hence, according to Tsing's theorem \cite{r.tsing} and  the fact that   $\left(F, \langle \cdot, \cdot \rangle_{A^{\prime}}\right)$ is a Hilbert space, we can conclude    that $W_{q, A^{\prime}}(\mathbf{T}'')$ is convex, and consequently  $W_{q,A^{\prime}}(T^{\prime})$ contains the segment joining $\langle T^{\prime}x_{2}, y_{2}\rangle_{A^{\prime}}$ and  $\langle T^{\prime}f_{2}, g_{2} \rangle_{A^{\prime}}$.  Therefore, $W_{q,A}(T)$ is convex.  
\item[\textbf{Case 2:}] If  $T(\mathcal{N}(A))\not\subseteq \mathcal{N}(A)$, then there exists $x_{1}\in \mathcal{N}(A)$ such that $Tx_1\not \in \mathcal{N}(A)$. This implies that $\Vert Tx_1 \Vert_{A} \neq 0$. By using   the decomposition  $\mathcal{H}=\mathcal{N}(A) \oplus \overline{\mathcal{R}(A)}$, we can infer that $\Vert Tx_1 \Vert_{A}=\Vert PTx_1 \Vert_{A}$. In particular, we get that  $\Vert PTx_1 \Vert_{A} \neq 0$. Put $y_{2}=\frac{PTx_1}{\Vert PTx_1 \Vert_{A}}$. Clearly $y_{2} \in \overline{\mathcal{R}(A)}\setminus \{ 0 \}$, so by applying Corollary \ref{corl4} to the pre-Hilbertien  space  $\left(\overline{\mathcal{R}(A)}, \langle \cdot, \cdot \rangle_{A}\right)$, we can conclude that there exists $w\in  \overline{\mathcal{R}(A)}$ such that $\Vert w \Vert_{A}=1$ and  $\langle y_{2}, w \rangle_{A}=0$. By setting $x_{2}=qy_2+\sqrt{1-q^{2}}w$, we obtain that $(x_2, y_2) \in \mathbb{S}_{q, A}(\overline{\mathcal{R}(A)})$. From the equality \eqref{equ6}, it follows that
\begin{align*}
W_{q,A}(T)&\supseteq  \left \{ \langle T(\xi x_1), y_{2} \rangle_A + \langle Tx_2, y_2 \rangle_A :  \xi \in \mathbb{C} \right \} \\
&=  \left \{  \xi \left \langle Tx_1, \frac{PTx_1}{\Vert PTx_1 \Vert_{A}} \right \rangle_A + \langle Tx_2, y_2 \rangle_A :  \xi \in \mathbb{C} \right \} \\
&= \left \{  \xi \Vert PTx_1 \Vert_{A}  + \langle Tx_2, y_2 \rangle_A :  \xi \in \mathbb{C} \right \} \\
&=\mathbb{C}.
\end{align*}
Consequently, in this case, we obtain that $W_{q,A}(T)=\mathbb{C}$ which is convex.

\end{enumerate}

\end{proof}

\begin{remark}
It is widely known that if $\mathcal{H}$ is finite dimensional, then $W_{q}(T)$ is compact. However, this result does not hold in general for $W_{q,A}(T)$. Indeed, for an operator $T\in \mathcal{B(H)}$ such that $T(\mathcal{N}(A))\not\subseteq \mathcal{N}(A)$, it follows from the proof of Theorem \ref{thm21prin} that $W_{q,A}(T)=\mathbb{C}.$
\end{remark}

\begin{proposition}\label{prop23prin}
We assume that $\mathcal{H}$  is finite dimensional. If $A$ is invertible, then  $\mathbb{S}_{q, A}(\mathcal{H})$ is compact.
\end{proposition}

\begin{proof}
We can easily check that    $\mathbb{S}_{q, A}(\mathcal{H})$ is a closed subset of  $\mathbb{S}^{A}(0,1) \times \mathbb{S}^{A}(0,1)$,  where  $\mathbb{S}^{A}(0,1)$ is the   $A$-unit sphere given by 
\begin{equation*}
\mathbb{S}^{A}(0,1)=\{ x\in \mathcal{H} : \Vert x \Vert_{A}=1 \}.
\end{equation*}
On the other hand,   according to \cite[Proposition 2.2]{r.bkos}, the set   $\mathbb{S}^{A}(0,1)$ is compact, since $A$ is invertible.  Hence,  $\mathbb{S}_{q, A}(\mathcal{H})$ is also compact.

\end{proof}

\begin{theorem}
We assume that $\mathcal{H}$  is finite dimensional.  Let $T\in \mathcal{B(H)}$ such that $T(\mathcal{N}(A)) \subseteq \mathcal{N}(A),$ then   $W_{q,A}(T)$ is a compact subset of $\mathbb{C}.$
\end{theorem}

\begin{proof}
From the proof of Theorem \ref{thm21prin}, we can conclude  that  
\begin{equation}\label{equa6prin}
W_{q,A}(T)=\left\{ \langle T^{\prime}x_2, y_2 \rangle_{A}: (x_2,y_2)  \in \mathbb{S}_{q, A^{\prime}}\left(\overline{\mathcal{R}(A)}\right) \right\} ,
\end{equation}
 where $A^{\prime}=A\mid_{\overline{\mathcal{R}(A)}}$ and $T^{\prime}=PT\mid_{\overline{\mathcal{R}(A)}}$. On the other hand,  since  $\overline{\mathcal{R}(A)}$ is finite  dimensional and the positive operator $A^{\prime} $ is invertible, then it follows from  Proposition \ref{prop23prin} that  $\mathbb{S}_{q, A^{\prime}}\left(\overline{\mathcal{R}(A)}\right)$ is compact. Thus, according to \eqref{equa6prin}, we infer that  $W_{q,A}(T)$ is compact.
\end{proof}

We introduce now the following quantity.

\begin{definition}
Let $T\in \mathcal{B(H)}$. The $q$-$A$-numerical radius of $T$ is given by
\begin{equation*}
w_{q,A}(T)=\sup\left \{ \vert \xi \vert :  \xi \in  W_{q,A}(T) \right\}.
\end{equation*}
\end{definition}

A direct application of the Cauchy-Schwarz inequality gives $w_{q, A}(T) <+ \infty$ for every  $T\in \mathcal{B}_{A^{\frac{1}{2}}}(\mathcal{H}).$ On the other hand, we can easily show that   $w_{q, A}(\cdot)$ is a semi-norm on $ \mathcal{B}_{A^{\frac{1}{2}}}(\mathcal{H})$. In particular, it   satisfies the triangle inequality:
\begin{equation*}
w_{q,A}(T+S) \leq w_{q,A}(T) + w_{q,A}(S) \ \ (T,S \in \mathcal{B}_{A^{\frac{1}{2}}}(\mathcal{H})).
\end{equation*}

 It is generally interesting to study the equality case. In the next result, we provide a characterisation of the equality $ w_{q,A}(T+S) = w_{q,A}(T) + w_{q,A}(S)$  to  hold in $ \mathcal{B}_{A^{\frac{1}{2}}}(\mathcal{H})$.
\begin{theorem}
Let $T,S \in \mathcal{B}_{A^{\frac{1}{2}}}(\mathcal{H}).$ The following statements are equivalent: 
\begin{enumerate}
\item $w_{q,A}(T+S)=w_{q,A}(T)+w_{q,A}(S);$
\item there exists a sequence $\{ (x_n, y_n) \}$ of elements of $\mathbb{S}_{q,A}(\mathcal{H})$ such that 
\begin{equation*}
\lim_{n\to +\infty} \Re \Big( \langle y_n, Tx_n \rangle_{A}  \langle Sx_n, y_n  \rangle_{A} \Big ) = w_{q,A}(T)  w_{q,A}(S).
\end{equation*}
Here, $\Re(z)$ denotes the real part of any $z \in \mathbb{C}.$
\end{enumerate}
\end{theorem}

\begin{proof}
$(1)\Rightarrow (2).$ By the hypothesis, there exists a sequence $\{ (x_n, y_n) \}$ of elements of $\mathbb{S}_{q,A}(\mathcal{H})$ such that 
\begin{equation*}
\lim_{n\to +\infty} \Big\vert \langle Tx_n, y_n \rangle_{A} + \langle Sx_n, y_n  \rangle_{A} \Big\vert^{2} =\Big(w_{q,A}(T) + w_{q,A}(S)\Big)^{2}. 
\end{equation*}
For each $n\in \mathbb{N},$ we have
\begin{align*}
\Big\vert \langle Tx_n, y_n \rangle_{A} + \langle Sx_n, y_n  \rangle_{A} \Big\vert^{2} & = \Big\vert \langle Tx_n, y_n \rangle_{A} \Big\vert^{2}  + \Big\vert \langle Sx_n, y_n  \rangle_{A} \Big\vert^{2}  \\ 
& \hspace{2cm} +2 \Re\Big(\overline{ \langle Tx_n, y_n \rangle_{A} } \langle Sx_n, y_n  \rangle_{A} \Big ) \\
&=  \Big\vert \langle Tx_n, y_n \rangle_{A} \Big\vert^{2}  + \Big\vert \langle Sx_n, y_n  \rangle_{A} \Big\vert^{2}  \\ 
& \hspace{2cm} +2 \Re\Big( \langle y_n, Tx_n \rangle_{A}  \langle Sx_n, y_n  \rangle_{A} \Big ) \\
& \leq w_{q,A}^{2}(T) + w_{q,A}^{2}(S) + 2  w_{q,A}(T)  w_{q,A}(S) \\
&=  \Big(w_{q,A}(T) + w_{q,A}(S)\Big)^{2}.
\end{align*}
Then, taking limits when $n \to +\infty,$ we conclude that
\begin{equation*}
\lim_{n\to +\infty} \Re \Big( \langle y_n, Tx_n \rangle_{A}  \langle Sx_n, y_n  \rangle_{A} \Big ) = w_{q,A}(T)  w_{q,A}(S).
\end{equation*}
$(2) \Rightarrow (1).$ Pick a sequence $\{ (x_n, y_n) \}$ of elements of $\mathbb{S}_{q, A}(\mathcal{H})$ such that 
\begin{equation*}
\lim_{n\to +\infty} \Re \Big( \langle y_n, Tx_n \rangle_{A}  \langle Sx_n, y_n  \rangle_{A} \Big ) = w_{q,A}(T)  w_{q,A}(S).
\end{equation*}
We know that  
\begin{align*}
\Re \Big( \langle y_n, Tx_n \rangle_{A}  \langle Sx_n, y_n  \rangle_{A} \Big ) \leq  \Big\vert \langle Tx_n, y_n \rangle_{A} \Big\vert w_{q,A}(S) \leq w_{q,A}(T)  w_{q,A}(S),
\end{align*}
 whence $\displaystyle \lim_{n\to +\infty}  \Big \vert \langle  Tx_n, y_n \rangle_{A} \Big\vert =  w_{q,A}(T).$  By a similar argument, we infer that $\displaystyle \lim_{n\to +\infty}  \Big \vert \langle  Sx_n, y_n \rangle_{A} \Big\vert =  w_{q,A}(S).$ Using the fact that 
 \begin{align*}
 \Big\vert \langle Tx_n, y_n \rangle_{A} + \langle Sx_n, y_n  \rangle_{A} \Big\vert^{2} &= \Big\vert \langle Tx_n, y_n \rangle_{A} \Big\vert^{2}  + \Big\vert \langle Sx_n, y_n  \rangle_{A} \Big\vert^{2}
 \end{align*}
  and the fact that 
\begin{align*}
\lim_{n\to +\infty} \Re \Big( \langle y_n, Tx_n \rangle_{A}  \langle Sx_n, y_n  \rangle_{A} \Big ) = w_{q,A}(T)  w_{q,A}(S),
\end{align*}  
  we obtain  $w_{q,A}(T+S)=w_{q,A}(T)+w_{q,A}(S).$
\end{proof}

\section{Conclusion}

In conclusion, this paper successfully introduces and investigates the concept of the $q$-numerical range for tuples of operators in Hilbert spaces. The establishment of various inequalities for the $q$-numerical radius associated with operator tuples enhances our understanding of this concept. Furthermore, the definition of the $q$-numerical range for operators in semi-Hilbert spaces and the proof of its convexity contribute to the development of this field.

This work serves as a starting point for further research. In particular, it opens up possibilities for studying the concept of $q$-joint numerical range and radius in the context of semi-Hilbert spaces. Exploring these extensions could provide deeper insights into the behavior and properties of operators in this setting.


\begin{thebibliography}{99}
	
\bibitem{li2000convexity} Li, Chi-Kwong and Poon, Yiu-Tung, \textit{Convexity of the joint numerical range}, SIAM Journal on Matrix Analysis and Applications, \textbf{21}, (2000), 668--678.	

\bibitem{gutkin2004convexity} Gutkin, Eugene and Jonckheere, Edmond A and Karow, Michael, \textit{Convexity of the joint numerical range: topological and differential geometric viewpoints}, Linear algebra and its applications, \textbf{376}, (20004), 143--171.	

\bibitem{feki2020tuples} Feki, Kais, \textit{On tuples of commuting operators in positive semidefinite inner product spaces}, Linear algebra and its applications, \textbf{603}, (2020), 313--328.	

\bibitem{zamani2019numerical} Zamani, Ali, \textit{A-numerical radius inequalities for semi-Hilbertian space operators}, Linear algebra and its applications, \textbf{578}, (2019), 159--183.	
	
\bibitem{chien2007q} Chien, Mao-Ting and Nakazato, Hiroshi, \textit{The q-numerical radius of weighted shift operators with periodic weights}, Linear algebra and its applications, \textbf{422}, (2007), 198--218.	
	
\bibitem{li1998some} Li, Chi-Kwong and Nakazato, Hiroshi, \textit{Some results on the q-numerical}, Linear and Multilinear Algebra, \textbf{43}, (1998), 385--409.	


\bibitem{hirzallah2011numerical} Hirzallah, Omar and Kittaneh, Fuad and Shebrawi, Khalid, \textit{Numerical radius inequalities for certain 2$\times$ 2 operator matrices}, Integral Equations Operator Theory, \textbf{71}, (2011), 129--147.

\bibitem{li1994generalized} Li, Chi-Kwong and Mehta, Paras P and Rodman, Leiba, \textit{A generalized numerical range: the range of a constrained sesquilinear form}, Linear and Multilinear Algebra, \textbf{37}, (1994), 25--49.

\bibitem{r.lacg} M. L.  Arias, G. Corach, M.C. Gonzalez, \textit{Metric properties of projections in semi-Hilbertian spaces}, Integral Equations Operator Theory, \textbf{62}, (2008), 11--28.


\bibitem{li2009joint}Li, Chi-Kwong and Poon, Yiu-Tung, \textit{The joint essential numerical range of operators: convexity and related results}, Studia Math, \textbf{194}, (2009), 91--104.

\bibitem{r.lacg1} M. L.  Arias, G. Corach, M.C. Gonzalez, \textit{Partial isometries in semi-Hilbertian spaces},  Linear Algebra Appl., \textbf{428}, (2008), 1460--1475.


\bibitem{r.bkos} H. Baklouti, K. Feki,  O.A.M. Sid Ahmed, \textit{Joint numerical ranges of operators in semi-Hilbertian spaces}, Linear Algebra Appl., \textbf{555}, (2018), 266--284.

%
%
%

\bibitem{r.dog} R. G. Douglas, \textit{On majorization, factorization, and range inclusion of operators on Hilbert space}, Proc. Amer. Math. Soc., \textbf{17}, (1966), 413--415.

%

\bibitem{r.kf6} K. Feki, \textit{A note on the $A$-numerical radius of operators in semi-Hilbert spaces}, Arch. Math., \textbf{115}, (2020), 535--544.






\bibitem{r.vm} V. Müller, \textit{Spectral theory of linear operators}, Second edition, Birkhäuser Basel, (2007).

%



\bibitem{r.tsing} N. K. Tsing, \textit{ The constrained bilinear form and the $C$-numerical range}. Linear Algebra Appl., \textbf{56}, (1984), 195--206.


\bibitem{r.wagschal} C. Wagschal, \textit{Topologie et analyse fonctionnelle}, Collection Méthodes : mathématiques : licence master agrégation écoles d'ingénieurs, Hermann, (2012).




\bibitem{baklouti2018joint} H. Baklouti, k. Feki, Ahmed, O.A.M.S. Ahmed, \textit{Joint numerical ranges of operators in semi-Hilbertian spaces}, Linear Algebra Appl., \textbf{555}, (2018), 266--284.

\bibitem{buoni1978joint} j. Buoni and B. Wadhwa, \textit{On joint numerical ranges}, Pacific J. Math., \textbf{77}(2), (1978), 303--306.	

\bibitem{gustafson1997numerical} K.E. Gustafson, and D.K.M. Rao,  K.E. Gustafson, and D.K.M. Rao, \textit{Numerical range}, Springer, 1997.	

\bibitem{gau2021numerical} H.L. Gau and P.Y. Wu, \textit{Numerical ranges of Hilbert space operators}, Cambridge University Press \textbf{179}, (2021).	

\bibitem{moore1969adjoints} R.T. Moore, \textit{Adjoints, numerical ranges, and spectra of operators on locally convex spaces}, 1969.	

\bibitem{chan2018note} J.T. Chan, \textit{A note on the boundary of the joint numerical range}, Taylor \& Francis, \textbf{66}(4), (2018), 821--826.

\bibitem{cho1981boundary} M. Ch{\=o} and M. Takaguchi, \textit{Boundary points of joint numerical ranges}, Pacific J. Math., \textbf{95}(1), (1981), 27--35.	

\bibitem{eliezer1968note} Eliezer, CJ, \textit{ A note on group commutators of 2$\times$ 2 matrices}. The American Mathematical Monthly, \textbf{75}(10), (1968), 1090--1091.

\bibitem{fakhri2022q} S.F. Moghaddam, A.K. Mirmostafaee, M. Janfada,  \textit{ q-Numerical radius inequalities for Hilbert space}. Linear Multilinear Algebra, (2023), 1--13.

\bibitem{drnovvsek2014joint} R. Drnov{\v{s}}ek,and V. M{\"u}ller, \textit{ On joint numerical radius II}. Linear Multilinear Algebra, \textbf{62}(9), (2014), 1197--1204.

\bibitem{ptak2019c} Ptak, Marek and Simik, Katarzyna and Wicher, Anna, \textit{ $ C $--normal operators}. arXiv preprint arXiv:1912.13369, (2019).

\end{thebibliography}
\end{document}